\documentclass[red,11pt,a4paper]{article}

\usepackage{cite}

\usepackage{amsmath}
\usepackage{amscd}
\usepackage{amssymb}
\usepackage{latexsym}
\usepackage{color}
\numberwithin{equation}{section}
\usepackage[normalem]{ulem}

\usepackage[colorlinks,
           linkcolor=blue,      
           anchorcolor=blue,  
           citecolor=red,       
            ]{hyperref}

  \newcommand{\vrp}{{\vr_+}}
    \newcommand{\vrm}{{\vr_-}}
      \newcommand{\bvrp}{{\tilde\vr_+}}
    \newcommand{\bvrm}{{\tilde\vr_-}}
      \newcommand{\gp}{{\gamma^+}}
      \newcommand{\gm}{{\gamma^-}}
      \newcommand{\Hp}{{H_+}}
      \newcommand{\Hm}{{H_-}}
      \newcommand{\Pp}{{p_+}}
      \newcommand{\Pm}{{p_-}}

\def\tilde{\widetilde}

\newcommand{\p}{\partial}

\newcommand{\BR}{{\tilde{R}}}
\newcommand{\BQ}{{\tilde{Q}}}

\newcommand{\Div}{\operatorname{div}}

\newcommand{\pf}{{\noindent\it Proof.~}}
\newcommand{\Ov}[1]{\overline{#1}}

\newcommand{\vf}{{\vc F}}

\newcommand{\vS}{{\bf S}}

\newcommand{\vr}{\varrho}

\newcommand{\vu}{\vc{u}}
\newcommand{\vv}{\vc{v}}

\newcommand{\vc}[1]{{\bf #1}}
\newcommand{\vcg}[1]{{\pmb #1}}

\newcommand{\Grad}{\nabla}

\newcommand{\pt}{\partial_{t}}
\newcommand{\ptb}[1]{\partial_{t}(#1)}

\newcommand{\dx}{{\rm d} {x}}

\newcommand{\dt}{{\rm d} t }

\newcommand{\dxdt}{\dx \,\dt}
\newcommand{\lr}[1]{\left( #1 \right)}
\newcommand{\intO}[1]{\int_{\Omega} #1 \ \dx}

\newcommand{\intTO}[1]{\int_0^T\!\!\!\! \int_{\Omega} #1 \ \dxdt}

\newcommand{\inttauO}[1]{\int_0^\tau\!\! \int_{\Omega} #1 \ \dxdt}

\newcommand{\inttauOB}[1]{ \int_0^\tau\!\!\!\! \int_{\Omega} \left( #1 \right) \ \dxdt}

\newcommand{\eq}[1]{\begin{equation}
\begin{split}
#1
\end{split}
\end{equation}}
\newcommand{\eqh}[1]{\begin{equation*}
\begin{split}
#1
\end{split}
\end{equation*}}

\newcommand{\R}{\mathbb{R}}
\newtheorem{thm}{Theorem}[section]
\newtheorem{lemma}[thm]{Lemma}

\newtheorem{df}[thm]{Definition}
\newtheorem{rmk}[thm]{Remark}

\newcommand{\CX}{{\mathcal X}}
\newcommand{\f}{\frac}

\title{\bf{Weak--strong uniqueness for bi-fluid compressible system with algebraic closure} }

\author{
Yang Li\thanks{School of Mathematical Sciences, Anhui University, Hefei--230601, People's Republic of China, E-mail: \texttt{lynjum@163.com}} \quad
M\'{a}ria Luk\'{a}\v{c}ov\'{a}-Medvid'ov\'{a}\thanks{Institute of Mathematics, Johannes Gutenberg-University Mainz
Staudingerweg 9, 55 128 Mainz, Germany,
E-mail:  \texttt{lukacova@uni-mainz.de}}  \quad
Milan Pokorn\'{y}\thanks{Faculty of Mathematics and
Physics, Mathematical Institute, Charles University, Sokolovsk\'{a} 83,
186 75 Prague 8, Czech Republic, E-mail: \texttt{pokorny@karlin.mff.cuni.cz}} \quad 
Ewelina Zatorska\thanks{Mathematics Institute, University of Warwick, Zeeman Building, Coventry CV4 7AL, United Kingdom, E-mail:  \texttt{ewelina.zatorska@warwick.ac.uk}}
}

\begin{document}
 \maketitle 

\abstract
We consider a real two-fluid system of compressible viscous fluids with a common velocity field and algebraic closure for the pressure law. The constitutive relation involves densities of both fluids through an implicit function. The existence of global-in-time finite energy weak solutions to this system is known since the work of  Novotn\'{y} and Pokorn\'{y} [\emph{Arch. Rational Mech. Anal.}, 2020]. On the other hand, existence of local-in-time strong solutions is due to Piasecki and Zatorska [\emph{J. Math Fluid Mech.}, 2022]. In this paper, we establish the weak--strong uniqueness principle  using the relative entropy method.
In sharp contrast to the two-phase model of Baer-Nunziato type, the volume fraction of phase $+$ obeys a transport equation with an additional nonlinear term. This gives rise to troublesome terms in the relative entropy inequality. We are able to close the estimate by making an elaborate use of the structure of the system.

\medskip

\vspace{4mm}


\tableofcontents

\section{Introduction} 
Multi-component compressible flow models (often termed multifluid or two-phase models) arise from interface-averaging procedures applied to microscopic balance laws. In the averaged description, volume fractions become additional unknowns and various closure assumptions are needed to close the macroscopic system. Classical modeling expositions can be found in the monographs of Ishii and Hibiki \cite{IsHi} and Drew and Passman \cite{DP}, as well as in more modern surveys of the mathematical literature on multifluid systems \cite{Br-vulcano, Leb25}.

A common modeling starting point is the two-velocity Baer–Nunziato framework \cite{BN}, from which various one-velocity reductions are obtained by relaxation and closure assumptions \cite{Kapila, Burtea}.

Two analytically distinct one-velocity regimes are:
\begin{itemize}
\item {\bf{Differential closure:}} the volume fraction 
$\alpha$ obeys a transport type equation (typical in one-velocity Baer–Nunziato-type models). This transport equation should include some relaxation term between the pressures, see \cite{Br-vulcano, Burtea}, but simplified models with pure transport equation have also been studied \cite{Novotny19,JN19}.

\item {\bf{Algebraic closure:}} the phases share a common velocity and are coupled by an implicit constitutive relation—most prominently equal pressure $p_+=p_-$. 
In this setting, $\alpha$ is not an independent transported scalar once one passes to the “transported” partial masses $R=\alpha\vrp$, $Q=(1-\alpha)\vrm$; instead, the implicit pressure relation generates an additional nonlinear structure; see below for more details.
\end{itemize}

The present paper analyzes the second class and considers the viscous compressible bi-fluid system in a bounded three-dimensional domain 
$\Omega$, where both phases are barotropic (power-law pressures) and the single momentum equation is driven by the mixture density $R+Q$. Our main contribution  is the proof of the weak–strong uniqueness principle: any finite-energy weak solution in the sense of Novotn\'y–Pokorn\'y \cite{NM20} coincides, on the lifespan of the strong solution, with the strong solution constructed by Piasecki–Zatorska \cite{PZ1}, if both solutions emanate from the same initial data.

For mono-fluid compressible Navier–Stokes equations, Lions-Feireisl theory \cite{Lions, Fe2002} relies crucially on monotonicity/compensated compactness mechanisms tied to a pressure depending on a single conserved quantity. In multifluid systems, the pressure typically depends on several quantities; even when these quantities satisfy the continuity equations, their weak convergences generally do not commute with nonlinear implicit pressure maps. This gives rise to new challenges. 

A first major step for algebraic closures in the semi-stationary regime is the global finite-energy existence theory for the two-fluid Stokes system by Bresch, Mucha, and Zatorska \cite{BMZ19}. Their analysis makes explicit the core obstruction: after reformulation, the pressure argument satisfies a transport equation with a nonlinear production term, preventing the standard effective-viscous-flux compactness route; this is precisely where the Bresch–Jabin compactness criterion was brought in \cite{BrJa}.

The fully time-dependent viscous bi-fluid model with algebraic pressure closure was then treated by Novotn\'y and Pokorn\'y, who proved the existence of global-in-time finite-energy weak solutions for large initial data \cite{NM20}. Their strategy was to transform the original variables into a two-densities Navier–Stokes-type system with a pressure of “complicated structure dependent on two variable densities”, and then to adapt the Lions–Feireisl method supplemented by DiPerna–Lions transport/renormalization ideas\cite{DiPL} refined in \cite{MaMiMuNoPoZa, VaWeYu}. Notably, all of these results required certain constraint allowing comparison between the two continuity equations, which has been subsequently removed by Wen \cite{Wen}. Other extensions of this theory include existence results for the compressible viscous non-resistive MHD model, for which we refer to Li and Sun \cite{LiSun1,LiSun2}.

For the sake of completeness, let us also recall that the existence of weak solutions to the inviscid two-fluid system with algebraic pressure closure was considered in \cite{LiZa}, where the solutions were constructed using the convex integration method. We also refer to \cite{Evje1,LSZ1, Evje2} for the existence results of weak solutions to related models in one space-dimension case.

The results from \cite{NM20} were further adapted  to treat differential closure (transport of $\alpha$) by Novotn\'y \cite{Novotny19}.  
For this system, Jin and Novotn\'y \cite{JN19} established the weak--strong uniqueness property of energy-bounded weak solutions, emphasizing the role of convexity of the associated Helmholtz potential. We refer also to \cite{KKNN} for problems with inflow-outflow boundary conditions and \cite{JKNN} for the weak--strong stability result but for the dissipative measure-valued solutions. Still, the assumption about convexity of the free Helmholtz energy with respect to the conservative variables seemed to be of paramount importance. 
\emph{Removal of this constraint is to date an important open problem that we resolved in the current paper.}

On the strong-solution side, local-in-time well-posedness for the compressible two-fluid algebraic-closure system was obtained by Piasecki and Zatorska  using the maximal-regularity techniques \cite{PZ1}. The existence of global-in-time strong  solutions for small perturbations of constant initial conditions for the nonconservative model is due to Evje, Wang and Wen \cite{Evje}.

Given the coexistence of global weak solutions and local strong solutions, a natural stability question is whether weak solutions are unique as long as a strong solution exists—the weak–strong uniqueness principle.
For the compressible two-fluid system, the only known results, as far as we are aware, are the aforementioned \cite{JN19,JKNN} and a conditional result by Li and Zatorska \cite{LZ22}.

For the mono-fluid compressible Navier–Stokes system, weak–strong uniqueness via the relative entropy method is nowadays classical \cite{FJN}. Beyond uniqueness, the relative entropy inequality has proven to have profound applications in the study of singular-limits problems, and stability of numerical schemes, see, for example, \cite{FKNZ,Luk, Choi}.
For the viscous compressible two-fluid system, the  related results are much more sparse. For the two-fluid model from  \cite{VaWeYu}, the incompressible and/or inviscid singular limits were considered in \cite{YaoZhu, KwonLi, YangCheng}, with the use of the relative entropy method. We also refer to the recent result of Fanelli, Kwon and Wr\'oblewska-Kaminska \cite{Fanelli26}, where the low Mach number limit  for the multi-component MHD system, in case of ill-prepared data, was proven.

The relative entropy inequality developed here has similar implications for singular limits as the pioneering result of Feireisl, Jin and Novotn\'y \cite{FJN} for mono-fluids: it supplies a robust stability mechanism that upgrades formal asymptotics to rigorous convergence, provided suitable strong limit systems are available. In particular, our method gives a direct route to justify various incompressible limits (at least for well-prepared data) studied formally and numerically, for example in  \cite{Bouchut, Leb25, Grenier}. Moreover, the method can also be extended to two-fluid systems with a differential closure relation \cite{Br-vulcano, Burtea}.

In sharp contrast to the differential-closure model \cite{JN19}, the algebraic constraint $\Pp=\Pm$ changes the evolution of the volume fraction to transport with an \emph{additional nonlinear term}. This phenomenon appears already in the Stokes analysis of Bresch–Mucha–Zatorska \cite{BMZ19} and is precisely what breaks the most direct relative-entropy estimates: extra terms arise that do not cancel as in the pure-transport case. 
 Our proof  introduces an additional mechanism to control the defect in the volume fraction generated by the algebraic closure. More concretely, we exploit the specific implicit structure of the pressure law to derive a stability estimate for the volume fractions, which is compatible with the relative energy inequality, and then we conclude with a Gr\"onwall argument. 
\emph{It should be emphasized that, in sharp contrast with Jin--Novotn\'{y} \cite{JN19}, no positive bound for $\alpha$ from below and away from $1$ is required in our analysis.}

The paper is organized as follows. In Section \ref{Sec:2}  after recalling the model and the notion of bounded-energy weak solutions, we summarize the known existence theory for global weak solutions and local strong solutions. We then derive the relative energy inequality, highlighting the additional nonlinear terms caused by the evolution law for the volume fraction in Section \ref{Sec:3}. Next, in Section \ref{Sec:4} we establish the relative entropy inequality with the strong solution as a reference state and estimate the resulting remainder terms using the structure induced by the algebraic closure. Finally, in Section \ref{Sec:5}, we combine these estimates to prove the weak–strong uniqueness principle, which is the main achievement of our paper.

\section{Formulation of the problem}\label{Sec:2}
 Our system of equations reads:
\begin{subequations}\label{spec}
\begin{align}
&\ptb{\alpha\vrp}
+
\Div(\alpha\vrp\vu)=0,\label{spec_1}\\
&\ptb{(1-\alpha)\vrm}
+
\Div((1-\alpha)\vrm\vu)=0,\label{spec_2}\\
&\ptb{(\alpha\vrp
+(1-\alpha)\vrm)\vu}
+\Div((\alpha\vrp+(1-\alpha)\vrm)\vu\otimes\vu)\nonumber\\
&\qquad\qquad\qquad\qquad\qquad
+
\Grad (\alpha p_{+}(\vrp) + (1-\alpha)p_{-}(\vrm))=\Div \vS(\Grad\vu),\label{spec_3}\\
& p_+(\vrp) = p_- (\vrm).\label{spec_4}
\end{align}
\end{subequations}
In the above system, we denoted $\vS=\vS(\Grad\vu)$ to be the usual Newtonian stress tensor given by 
\begin{align*}
  \vS(\Grad\vu)= \mu \Big(\nabla \vu +(\nabla \vu)^T\Big) + \lambda \Div \vu \,{\mathbf I}, 
\end{align*}
where we assume that $\mu>0$ and $2\mu + 3 \lambda \geq 0$.

By $p_+$, $p_-$ we denote the internal barotropic pressures for each fluid with the explicit form:
\eq{
\label{CP}
p_+=( {\vr_+} )^{\gamma^+}, \quad
p_-=( {\vr_-} )^{\gamma^-}, 
}
where $\gamma^\pm>1$ are given adiabatic exponents. 
The system is endowed with the boundary condition
\begin{equation} \label{eq1.2bi}
\vu(t,x)\Big|_{\partial\Omega} = \vc{0}
\end{equation}
in $(0,T)\times \partial \Omega$,  where $\Omega$ is a sufficiently smooth bounded domain of $\R^3$ and $T>0$ is arbitrarily large, however, finite. 
We consider the initial conditions in $\Omega$
\begin{equation} \label{eq1.3bi}
\begin{aligned}
\alpha \vr_+(0,x) &= \alpha _0\vr_{+,0}(x)=:R_0(x), \\
(1-\alpha )\vr_-(0,x) &= (1-\alpha _0) \vr_{-,0}(x)=:Q_0(x), \\
(\alpha \vr_++(1-\alpha )\vr_-)\vu(0,x)&= 
(\alpha _0\vr_{+,0}+(1-\alpha _0)\vr_{-,0})\vu_0(x)=:\vc m_0.
\end{aligned}
\end{equation}

The main result of this paper reads

\begin{thm}\label{thm_entropy}
    Let $(\alpha , \vr_-,\vr_+,\vu)$ be a global weak solution to the system \eqref{spec}--\eqref{eq1.3bi} obtained in \cite{NM20}, and let $(\beta , \bvrm,\bvrp,\vv)$ be the strong solution obtained in
    \cite{PZ1} emanating from the same initial data. Then the two solutions coincide on the time interval $(0,T)$ of a lifespan of the regular solution.
\end{thm}

To prove this theorem, we use the method of relative entropy, inspired by Jin and Novotn\'{y} \cite{JN19}. We will show that the following 
relative entropy functional
\eq{
\label{rel_entr}
&{\cal E}(\alpha,R,Q,\vu \,| \, \beta,\BR,\BQ,\vv)(t)\\
=&
\int_{\Omega}
\frac12(R+Q) |\vu-\vv|^2 (t,x) \dx +\int_{\Omega}
\frac12(\alpha-\beta)^2 (t,x) \dx 
\\
&
+
\int_{\Omega}
\alpha\Big(\Hp(\vrp)-\Hp'(\bvrp)(\vrp-\bvrp)-\Hp(\bvrp)
\Big) 
(t,x) \dx \\
&
+
\int_{\Omega}
(1-\alpha)\Big( \Hm(\vrm)-\Hm'(\bvrm)(\vrm-\bvrm)-\Hm(\bvrm)
\Big) 
(t,x) \dx,
}
where for brevity we set 
\eq{\label{RQ_def}
R=\alpha\vrp,\quad Q=(1-\alpha)\vrm,\quad
\BR=\beta\bvrp,\quad \BQ=(1-\beta)\bvrm,
}
and
\begin{align*}
H_+(\vrp)=\frac{1}{\gp-1}\vrp^\gp,\quad \Hm(\vrm)=\frac{1}{\gm-1}\vrm^\gm,
\end{align*}
remains small, provided it is small initially, and that the regular solution $(\beta, \BR, \BQ,\vv)$ emanating from the same initial data exists.


\subsection{Global weak solutions}

We now recall the definition of a weak solution to system \eqref{spec}, see Definition 2 from Novotn\'{y} and Pokorn\'{y} \cite{NM20}, and the corresponding existence Theorem 2. For notational convenience, we set $I:=(0,T)$.

\begin{df}\label{d1bi}
The quadruple $(\alpha , \vr_-,\vr_+,\vu)$ is called a bounded energy weak solution to the problem \eqref{spec}--\eqref{eq1.3bi} in $I\times \Omega$, if 
\begin{align*}
 & 0\le\alpha \le 1, \quad \vr_\pm \geq 0 \quad \text{a.a. in } \, I\times \Omega, \\
 & \vr_\pm \in L^\infty(I;L^{1}(\Omega)), \quad 
 \vu \in L^2(I;W^{1,2}_0(\Omega;\R^3), \\
 & (\alpha \vr_++(1-\alpha )\vr_- )|\vu|^2 \in L^\infty(I;L^1(\Omega)), 
 \quad 
 \Pm(\vrm)= \Pp(\vrp) \in L^1(I\times \Omega),
\end{align*}
and
\begin{itemize}
\item Continuity equations \eqref{spec_1} and \eqref{spec_2} are satisfied in the weak sense
\begin{equation} \label{eqR}
\begin{aligned}
\intTO{\big(\alpha\vrp \partial_t \psi + \alpha\vrp \vu \cdot \Grad \psi\big)}  + \intO{R_0 \psi(0,\cdot)} =0,
\end{aligned}
\end{equation}
\begin{equation} \label{eqQ}
\begin{aligned}
\intTO{\big((1-\alpha)\vrm \partial_t \psi + (1-\alpha)\vrm \vu \cdot \Grad \psi\big)}  + \int_\Omega Q_0 \psi(0,\cdot) \dx =0,
\end{aligned}
\end{equation}
for any $\psi \in C^1_c([0,T) \times \Ov{\Omega})$.
\item Momentum equation \eqref{spec_3} is satisfied in the weak sense
\begin{equation} \label{equ}
\begin{aligned}
&\intTO{\big((\alpha\vrp
+(1-\alpha)\vrm)\vu \cdot \partial_t \vcg{\varphi} + (\alpha\vrp
+(1-\alpha)\vrm)(\vu\otimes \vu): \Grad \vcg{\varphi} + p \Div \vcg{\varphi}\big)}\\
= &\intTO{ \vS(\Grad\vu):\nabla \vcg{\varphi}} - \intO{\vc{m}_0 \cdot \vcg{\varphi}(0,\cdot) } , 
\end{aligned} 
\end{equation}
where $p=\alpha p_{+}(\vrp) + (1-\alpha)p_{-}(\vrm)$, and
for any $\vcg{\varphi} \in C^1_c([0,T) \times \Omega;\R^3)$.
\item The energy estimate
\eq{\label{energy}
&
\intO{\lr{\frac12(\alpha\vrp
+(1-\alpha)\vrm)|\vu|^2+\alpha \Hp(\vrp)+(1-\alpha)\Hm(\vrm)}(\tau) } \\
&\qquad\qquad+\inttauO{\vS(\Grad\vu): \nabla \vu}\\
&\leq \intO{\lr{\frac12(\alpha\vrp
+(1-\alpha)\vrm)|\vu|^2+\alpha \Hp(\vrp)+(1-\alpha)\Hm(\vrm)}(0) }
}
holds for a.a. $\tau \in I$.
\end{itemize}
\end{df}

Under the algebraic closure for the pressure, Novotn\'y and Pokorn\'y \cite{NM20} (see Section 5 therein) have shown that the system \eqref{spec} admits global-in-time weak solutions.
\begin{thm} \label{t1bi}
Let $0\le\underline a<\overline a<\infty$, $\gamma_{BOG}^\pm= \min\{\frac 23 \gamma^\pm-1,\frac{\gamma^\pm}{2}\}$.
Let $G :=\gamma^+ + \gamma^+_{BOG}$
if $\underline{a} =0$ and $G:= \max\{\gamma^+ +\gamma^+_{BOG}, \gamma^- +\gamma^-_{BOG}\}$ if $\underline{a} >0$. Assume
\begin{equation}\label{gamma}
0<\gamma^-<\infty,\qquad  \gamma^+\ge \frac{9}{5}, \qquad \overline\Gamma< G,
\end{equation}
where
\begin{align*}
    \overline\Gamma :=
\left\{\begin{array}{c}
\max\{\gamma^+ -\frac{\gamma^+} {\gamma^-}+ 1,
\, \gamma^- +\frac{\gamma^-} {\gamma^+}  -\frac{\gamma^+} {\gamma^-}\} 
 \qquad \mbox{if $\underline a=0$} \\
\max\{\gamma^+ -\frac{\gamma^+} {\gamma^-}+ 1, 
\gamma^-
+\frac{\gamma^-} {\gamma^+} -1
 \} \qquad \mbox{if $\underline a>0$}
\end{array}
\right\}.
\end{align*}
Suppose that
\begin{align}
& 0\le\alpha _0\le 1, \qquad 
\underline a\alpha _0
\vr_{+,0}\le(1-\alpha _0)\vr_{-,0}\le \overline a\alpha _0
\vr_{+,0}, \qquad 
p_+(\vr_{+,0})=p_-(\vr_{-,0}), 
\label{Thm_NP} \\
& \vr_{+,0}\in L^{\gamma^+}(\Omega),\qquad 
{\frac{|\vc m_0|^2}{\alpha _0
\vr_{+,0}+(1-\alpha _0)\vr_{-,0}}\in L^1(\Omega).} \nonumber
\end{align}

Then problem \eqref{spec}--\eqref{eq1.3bi} admits at least one weak solution in the sense of Definition \ref{d1bi}. Moreover, $\alpha \vr_+$ belongs to the space $C_{\rm{ weak} }([0,T]; L^{\gamma^+} (\Omega))$ and $(1- \alpha )\vr_-$ belongs to the space $C_{\rm{ weak} }([0,T]; L^{q_{\gamma^+,\gamma^-}} (\Omega))$\footnote{Here, the notation is $q_{\theta_1,\theta_2}=\theta_1$ if $\theta_2<\theta_1$ and $q_{\theta_1,\theta_2}=\theta_2$ if $\theta_2 \geq \theta_1$.},
the vector field $(\alpha \vr_+ + (1- \alpha )\vr_-)\vu$ belongs to $C_{\rm{ weak} }([0,T]; L^r(\Omega;\R^3))$ for some $r>1$, and $p_\pm(\vr_\pm) \in L^r((0,T)\times\Omega)$ for some $r>1$.
\end{thm}

\subsection{Local-in-time strong solutions}

Here we recall the local well-posedness for our bi-fluid system with large initial data from the work of Piasecki and Zatorska\cite{PZ1}. 
To this purpose, we consider the following reformulation of the system \eqref{spec} with $R$ and $Q$ defined in \eqref{RQ_def}
\begin{subequations}\label{S}
\begin{align}
  &\pt R+ \Div(R\vu )=0,\label{SR}\\ 
  &\pt Q + \Div(Q\vu )=0 ,\label{ST}\\
&\ptb{(R+Q)\vu}+\Div((R+Q)\vu\otimes\vu)-\Div \vS(\Grad \vu) +\Grad Z^{\gamma^+}=\vc{0}.
\end{align}
\end{subequations} 

In the above system, the pressure $p$ is expressed in terms of $R,Q$. In fact, we have
\begin{equation}\label{pZ}
 p=P(R,Q)= {Z}^{\gamma_+},
\end{equation}
for $Z=Z(R,Q)$ such that 
\eq{\label{TZ}
Q = \lr{1-\frac{R}{Z} } Z^\gamma,\quad  \mbox{with}\quad \gamma=\frac{\gamma_+}{\gamma_-},
}
and
\eq{\label{RleqZ}R\leq Z.}
 Relation \eqref{pZ} identifies $Z$ as $\vrp$ and thus \eqref{RleqZ} can be used to deduce the existence of $\alpha$ such that $0\leq \alpha\leq 1$ and $R=\alpha Z$, see \cite{BMZ19} and also \cite{NM20} for a similar argument.

The existence of local-in-time strong solutions to the system \eqref{S} is stated below.

\begin{thm} \label{t:local}
Assume that $\Omega$ is a uniform $C^2$ domain and $2<p<\infty, \; 3<q<\infty, \; \frac{2}{p}+\frac{3}{q}<1$.
Assume, moreover, that $R_0,Q_0$ satisfy 
\begin{align*}
{R}_0 \ge 0, \quad {Q}_0 \ge 0, \quad R_0+Q_0\geq\kappa \;\; \textrm{for some} \; \kappa>0
\end{align*} 
and $\vu_0$ satisfies the compatibility condition
\begin{align*}
\vu_0|_{\partial\Omega}=\mathbf{0}.
\end{align*} 

Then for any $L>0$ there exists $T>0$ such that if 
\begin{align*}
\|\nabla R_0\|_{L^q(\Omega)}+\|\nabla Q_0\|_{L^q(\Omega)}+\|\vu_0\|_{B^{2-2/p}_{q,p}(\Omega)} \leq L,
\end{align*} 
then \eqref{S} supplemented by the boundary condition \eqref{eq1.2bi}, the initial conditions \eqref{eq1.3bi} and the relation \eqref{TZ} admits a unique solution $(R,Q,\vu)$  in $(0,T)$ with the estimate  
\begin{align*}
\|(R,Q,\vu)\|_{\CX(T)} \leq CL, \qquad \int_0^T \|\nabla \vu(t)\|_{L^\infty(\Omega)} \, \dt 
\leq \delta,    
\end{align*} 
where 
\begin{align*}
\|(g,h,\vf)\|_{{\cal{X}}(T)}:=
\|\vf\|_{L^p(0,T;W^{2,q}(\Omega))}+ \| \vf_t\|_{L^p(0,T;L^q(\Omega))}+\|(g,h)\|_{W^{1,p}(0,T;W^{1,q}(\Omega))}
\end{align*} 
and $\delta$ is a small positive constant.  
\end{thm}

\begin{rmk}
    Although Theorem \ref{t:local} provides a local well-posedness for the reformulation of the system \eqref{spec} in which $R=\alpha\vrp$, $Q=(1-\alpha)\vrm$ and $p=p(R,Q)$, due to the algebraic closure \eqref{spec_4}, $\alpha$ can be recovered as $\alpha(R,Q)$, see the discussions in Bresch--Mucha--Zatorska \cite{BMZ19}. Hence, the original system \eqref{spec} admits a unique local strong solution for large initial data. 
\end{rmk}

\begin{rmk}
Theorem \ref{t:local} provides existence of solutions in the regularity class required for the purposes of the relative entropy argument. In particular, one can verify that
\begin{align*}
&\|\pt\vu+\vu\cdot \nabla \vu\|_{L^2(0,T; L^\infty(\Omega))}+\|\vu\|_{L^2(0,T; W^{1,\infty}(\Omega))}\leq C, \\
&\|\vrp\|_{L^\infty(0,T;W^{1,\infty}(\Omega))}+\|\vrm\|_{L^\infty(0,T;W^{1,\infty}(\Omega))}\leq C,
\end{align*}
as well as $\alpha$ such that
\begin{align*}
\|\alpha\|_{L^\infty(0,T; W^{1,\infty}(\Omega))}\leq C.
\end{align*}

\end{rmk}

\section{Relative entropy inequality--preliminaries}\label{Sec:3}

In the following parts of this section, we will aim at writing an equation for the evolution of the relative energy for $\beta,\BR,\BQ$ and $\vv$ sufficiently regular. 
The distinct difference with respect to the two-fluid model of Baer-Nunziato type \cite{JN19} is that in our case the equations for $\vr_+$ and $\vr_-$ are not just continuity equations, but are more complicated. 
This follows from the fact that $\alpha$ does not satisfy a pure transport equation, see \eqref{equa_alpha} for more details.

\subsection{Derivation of the equations for \texorpdfstring{$\alpha$, $\vr_+$ and $\vrm$}{}}
To compute this relative entropy inequality, we will need the equations for $\alpha$, $\vrp$ and $\vrm$.

Formally, we can  calculate
\begin{align*}
 \p_t \alpha + \vu \cdot \Grad \alpha    &  = \p_{R} \alpha \p_t R+ \p_{Q} \alpha \p_t Q + \p_{R} \alpha \vu \cdot \Grad R + \p_{Q} \alpha   \vu \cdot \Grad Q    \\
    &   =  \p_{R} \alpha ( \p_t R + \vu \cdot \Grad R) +
     \p_{Q} \alpha (\p_t Q+\vu \cdot \Grad Q   ) \\
    &   = -  \p_{R} \alpha R \Div \vu - \p_{Q} \alpha Q \Div \vu ;
\end{align*}
whence
\begin{align*}
\p_t \alpha + \vu \cdot \Grad \alpha +
(\p_{R} \alpha R+ \p_{Q} \alpha Q  ) 
\Div \vu=0 . 
\end{align*}
Denoting
\eq{   \label{g_frac}
\gamma=\gamma^+/\gamma^->0
}
and using the equality of pressures \eqref{spec_4} and the constitutive relations \eqref{CP}, we can find that
\eq{  \label{rel_RQ}
R^\gamma(1-\alpha)=Q\alpha^\gamma.
}
Hence, applying $\partial_Q$ and $\partial_R$ on both sides, we obtain
\eq{   \label{deriva_RQ}
\partial_Q\alpha=\frac{-\alpha^\gamma}{Q\gamma\alpha^{\gamma-1}+R^\gamma},\qquad
\partial_R\alpha=\frac{\gamma R^{\gamma-1}(1-\alpha)}{Q\gamma\alpha^{\gamma-1}+R^\gamma}
.  
}
Based on \eqref{rel_RQ}--\eqref{deriva_RQ}, it holds 
\begin{align*}
   \p_{R} \alpha R+ \p_{Q} \alpha Q  &    =    
    \frac{ \gamma R^{\gamma}(1-\alpha)    }{ Q \gamma \alpha^{\gamma-1}+R^{\gamma}    }   +     \frac{ -Q \alpha^{\gamma}   }{ Q \gamma \alpha^{\gamma-1}+R^{\gamma}    }        \\
    &   =   \frac{ \gamma R^{\gamma}(1-\alpha) -Q \alpha^{\gamma}     }{ Q \gamma \alpha^{\gamma-1}+R^{\gamma}    }   \\
    &  = \frac{  (\gamma-1) Q \alpha^{\gamma}    }{  \frac{ \gamma }{\alpha} Q \alpha^{\gamma}  + \frac{ 1}{1-\alpha  } Q \alpha^{\gamma}        }                  \\
    & =   \frac{  \gamma-1 }{  \frac{ \gamma }{\alpha} + \frac{ 1}{1-\alpha  }    }     \\
    & = \frac{ \gamma-1  }{ \gamma(1-\alpha) +\alpha } \alpha (1-\alpha) .  
\end{align*}
Thus, we obtain 
\eq{    \label{equa_alpha}
\p_t \alpha + \vu \cdot \Grad \alpha + \frac{ \gamma-1  }{ \gamma(1-\alpha) +\alpha } \alpha (1-\alpha)
\Div \vu=0. 
}

Equation \eqref{equa_alpha} is, of course, not exactly a pure  transport equation, but we should notice that because $0\leq\alpha\leq1$, the coefficient 
\begin{align*}
   \omega_\alpha: =\frac{ \gamma-1  }{ \gamma(1-\alpha) +\alpha } \alpha (1-\alpha) 
\end{align*}
is also bounded in $L^\infty((0,T)\times\Omega)$. 
Indeed, $\gamma(1-\alpha) +\alpha$ has a strictly positive lower bound, namely $\gamma(1-\alpha) +\alpha \geq \min\{1,\gamma \}>0$. Thus $|\omega_\alpha|\leq \frac{|\gamma-1|}{4 \min\{1,\gamma \}} < \frac{ \gamma+1 }{ \min\{1,\gamma \}  }$.

In a similar manner, we obtain
\eq{    \label{equa_1-alpha}
\p_t (1-\alpha) + \vu \cdot \Grad (1-\alpha) - \frac{ \gamma-1  }{ \gamma(1-\alpha) +\alpha } \alpha (1-\alpha)
\Div \vu=0 . 
}

From equations \eqref{equa_alpha}, and \eqref{spec_1}, \eqref{spec_2}, we now deduce
\eq{\label{alpha_from_R}
\alpha\pt\vrp+  \vrp \pt\alpha  +\Div(\vrp\vu)\alpha+ \vrp   (\vu\cdot\Grad\alpha)  =0,
}
therefore, using \eqref{equa_alpha}, we obtain
\begin{align*}
\pt\vrp+\Div(\vrp\vu)=\vrp \frac{ \gamma-1  }{ \gamma(1-\alpha) +\alpha }  (1-\alpha)
\Div \vu
\end{align*}
and
\begin{align*}
\pt\vrm+\Div(\vrm\vu)=-\vrm \frac{ \gamma-1  }{ \gamma(1-\alpha) +\alpha }  \alpha
\Div \vu.
\end{align*}

 In order to derive equations \eqref{equa_alpha} and \eqref{equa_1-alpha} rigorously, it suffices to follow the arguments from Section 8.1. in \cite{MaMiMuNoPoZa}, see also Lemma 2.4 in \cite{BMZ19}. In particular, we have the following result.
\begin{lemma}\label{Lemma_equiv}
Let $\vu\in L^2(0,T; W^{1,2}_0(\Omega;\R^3))$, $R=\alpha\vrp\in L^\infty(0,T; L^{2}(\Omega))$, and
$Q=(1-\alpha)\vrm \in L^\infty(0,T; L^{2}(\Omega))$,
and let $(R, Q,\vu)$ solve \eqref{SR} and \eqref{ST} in the sense of distributions. Then $Z$ defined by \eqref{TZ} belongs to $L^\infty(0,T; L^{2}(\Omega))$ and satisfies
\eq{\label{SZ}
\pt Z+\Div(Z\vu)+\frac{(1-\gamma) (Z-R)Z}{ \gamma(Z-R)+R}\Div\vu=0,}
in the sense of distributions.

{Conversely}, let $(R,Z,\vu)$ solve \eqref{SR} and \eqref{SZ} in the sense of distributions. Then $Q$ defined by \eqref{TZ} satisfies \eqref{ST} in the sense of distributions. 
\end{lemma}
The proof of  Lemma \ref{Lemma_equiv} can be found in \cite{BMZ19}, see Lemma 2.4  and is based on the DiPerna-Lions renormalization theory \cite{DiPL}. 

\begin{rmk}
    Let us remark that the assumptions of Lemma \ref{Lemma_equiv} are satisfied. Indeed, due to the assumption \eqref{gamma} and restriction of the initial data \eqref{Thm_NP}, one can show that the solutions to the continuity equations for $R$ and $Q$ remain in $L^\infty(0,T; L^{2}(\Omega))$.
\end{rmk}

We can now use Lemma \ref{Lemma_equiv} to identify the equations satisfied by $\alpha$, $\vrp$ and $\vrm$.
\begin{lemma}
 Let the assumptions of Lemma \ref{Lemma_equiv} be satisfied. Then $\vrp$ belongs to $L^\infty(0,T; L^{2}(\Omega))$ and it satisfies 
\eq{\label{rhop}
\pt\vrp+\Div(\vrp\vu)=\vrp \frac{ \gamma-1  }{ \gamma(1-\alpha) +\alpha }  (1-\alpha)
\Div \vu}
in the sense of distributions. Moreover, $\alpha$ belongs to $L^\infty((0,T)\times\Omega)$ and satisfies
\eq{    \label{equa_alpha1}
\p_t \alpha + \vu \cdot \Grad \alpha + \frac{ \gamma-1  }{ \gamma(1-\alpha) +\alpha } \alpha (1-\alpha)
\Div \vu=0 
}
in the sense of distributions.
\end{lemma}
\pf Let us note that equation \eqref{rhop} follows directly from \eqref{SZ} after substituting $Z=\vrp$ and $R=\alpha\vrp$. The equation for $\alpha$ is derived in the same way thanks to the decomposition \eqref{alpha_from_R}. The bound $0\leq\alpha\leq 1$ follows from the existence result, see Theorem \ref{t1bi}. $\Box$

\subsection{Derivation of relative entropy inequality}
We now derive the estimate of a part of the relative entropy functional from \eqref{rel_entr}, namely
\eq{
\label{rel_entr_tilde}
&\bar{\cal E}(R,Q,\vu \,| \, \BR,\BQ,\vv)(\tau)\\
&=
\int_{\Omega}
\frac12(R+Q) |\vu-\vv|^2 (\tau,x) \dx 
\\
&
\quad+
\int_{\Omega}
\alpha\Big(\Hp(\vrp)-\Hp'(\bvrp)(\vrp-\bvrp)-\Hp(\bvrp)
\Big) 
(\tau,x) \dx \\
&
\quad+
\int_{\Omega}
(1-\alpha)\Big( \Hm(\vrm)-\Hm'(\bvrm)(\vrm-\bvrm)-\Hm(\bvrm)
\Big) 
(\tau,x) \dx,
}
where $(\beta,\bvrp,\bvrm,\vv)$ are smooth test functions with $\vv|_{\p \Omega}=\mathbf{0}$, and we set $\BR=\beta \bvrp, \,\, \BQ=(1-\beta) \bvrm$.

We now expand
\eq{\label{rel_form}
&\bar{\cal E}(R,Q,\vu \,| \, \BR,\BQ,\vv)\Big|_0^\tau\\
&=\intO{\lr{\frac12(R+Q)|\vu|^2+\alpha \Hp(\vrp)+(1-\alpha)\Hm(\vrm)}}\Big|_0^\tau\\
& \quad 
+\intO{\frac12(R+Q)|\vv|^2}\Big|_0^\tau\\
& \quad 
-\intO{
(R+Q)\vu\cdot\vv}\Big|_0^\tau\\
& \quad 
-\intO{\alpha\vrp \Hp'(\bvrp)
}\Big|_0^\tau+\intO{\alpha\lr{\Hp'(\bvrp)\bvrp-\Hp(\bvrp)}}\Big|_0^\tau\\
& \quad 
-\intO{(1-\alpha)\vrm \Hm'(\bvrm)
}\Big|_0^\tau+\intO{(1-\alpha)\lr{\Hm'(\bvrm)\bvrm-\Hm(\bvrm)}}\Big|_0^\tau
}
and we will treat the right-hand side of \eqref{rel_form} line by line.


\medskip

\noindent
{\it Step 1. } 
The first line is estimated using the energy estimate \eqref{energy}, from which it follows that
\eq{\label{energy_new}
&
\intO{\lr{\frac12(R+Q)|\vu|^2+\alpha \Hp(\vrp)+(1-\alpha)\Hm(\vrm)}}\Big|_0^\tau
\leq 
-\inttauO{\vS(\Grad\vu): \nabla \vu} . 
}


\noindent
{\it Step 2. } We sum the equations for $Q$ and $R$ and test the obtained result by $\frac12|\vv|^2$
\eq{\label{rhov2}
\intO{\frac12(R+Q)|\vv|^2}\Big|_0^\tau =
\inttauO{(R+Q)\vv\cdot\lr{\pt\vv+(\vu \cdot \Grad) \vv}}.
}
Note that \eqref{rhov2} is not precisely the weak formulation; however, it can be easily deduced by considering the test function $\frac 12 |\vv|^2 \psi_\varepsilon(t)$ with $\psi_\varepsilon \in C^\infty_c([0,\tau))$, $\psi_\varepsilon \to {\mathbf 1}_{[0,\tau]}$ uniformly for $\varepsilon \to 0_+$. The same also applies below.  

\noindent
{\it Step 3. } We take $\vv$ as a test function in the momentum equation \eqref{equ}, and integrate by parts to get
\eq{\label{uv}
-\intO{(R+Q)\vu\cdot\vv}\Big|_0^\tau
=&-\inttauOB{(R+Q)\vu\cdot\lr{\pt\vv+(\vu\cdot\Grad)\vv}} \\
& \qquad \qquad 
-
\inttauOB{p\Div\vv-\vS(\Grad\vu):\Grad\vv}.
}


\noindent
{\it Step 4.} Finally, we test the equation for $R$ by $\Hp'(\bvrp)$ and the equation for $Q$ by $\Hm'(\bvrm)$ to get
\eq{   \label{step4+}
-\inttauO{\alpha\vrp\Hp'(\bvrp)}\Big|_0^\tau=-\inttauO{\alpha\vrp\lr{\pt+\vu\cdot\Grad} \Hp'(\bvrp)}
}
\eq{   \label{step4-}
-\inttauO{(1-\alpha)\vrm\Hm'(\bvrm)}\Big|_0^\tau=-\inttauO{(1-\alpha)\vrm\lr{\pt+\vu\cdot\Grad} \Hm'(\bvrm)}.
}

\noindent
{\it Step 5.} Finally, using $\Hp'(\bvrp)\bvrp-\Hp(\vr)=\Pp(\bvrp)$, we obtain
\begin{align*}
\intO{\alpha\lr{\Hp'(\bvrp)\bvrp-\Hp(\bvrp)}}\Big|_0^\tau=\intO{\alpha\Pp(\bvrp)}\Big|_0^\tau.
\end{align*}
In a similar manner
\begin{align*}
\intO{(1-\alpha)\lr{\Hm'(\bvrm)\bvrm-\Hm(\bvrm)}}\Big|_0^\tau
=
\intO{(1-\alpha)\Pm(\bvrm)}\Big|_0^\tau
.
\end{align*}

\noindent
{\it Step 6.} Summing the above expressions, we obtain all together:
\eq{      \label{sumstep1}
&\bar{\cal E}(R,Q,\vu\,| \,\BR,\BQ,\vv)\Big|_0^\tau+\inttauO{\vS(\Grad\vu):\Grad(\vu-\vv)}\\
&\leq \inttauO{(R+Q)(\vv-\vu)\cdot\lr{\pt\vv+(\vu \cdot \Grad) \vv}}-\inttauO{p\Div\vv}\\
&-\inttauO{\alpha\vrp(\pt+\vu\cdot\Grad)\Hp'(\bvrp)}\\
&-\inttauO{(1-\alpha)\vrm(\pt+\vu\cdot\Grad)\Hm'(\bvrm)}\\
&+\intO{\lr{\alpha\Pp(\bvrp)+(1-\alpha)\Pm(\bvrm)}}\Big|_0^\tau.
}
Adding to both sides $-\inttauO{\vS(\Grad\vv):\Grad(\vu-\vv)}$, we get
\eq{      \label{sumstep1b}
&\bar{\cal E}(R,Q,\vu\,| \,\BR,\BQ,\vv)\Big|_0^\tau+\inttauO{\vS(\Grad\vu-\Grad\vv):\Grad(\vu-\vv)}\\
&\leq \inttauO{(R+Q)(\vv-\vu)\cdot\lr{\pt\vv+(\vu \cdot \Grad) \vv}+ \vS(\Grad\vv):\Grad(\vv-\vu)}\\
&-\inttauO{p\Div\vv}\\
&-\inttauO{\alpha\vrp(\pt+\vu\cdot\Grad)\Hp'(\bvrp)}\\
&-\inttauO{(1-\alpha)\vrm(\pt+\vu\cdot\Grad)\Hm'(\bvrm)}\\
&+\intO{\lr{\alpha\Pp(\bvrp)+(1-\alpha)\Pm(\bvrm)}}\Big|_0^\tau.
}

{\it Step 7.} Now observe that because $\bvrp\Grad\Hp'(\bvrp)=\Grad\Pp(\bvrp)$, we have
\begin{align*}
&-\inttauO{(\alpha\Pp(\bvrp)+(1-\alpha)\Pm(\bvrm))\Div\vv}\\
& \quad 
=\inttauO{\Grad\Pp(\bvrp)
\cdot \vv}\\
&\quad
=\inttauO{ \Big(\alpha\Grad\Pp(\bvrp)+(1-\alpha)\Grad\Pm(\bvrm) \Big)
\cdot \vv}
\\
& \quad =\inttauO{\alpha\bvrp\Grad \Hp'(\bvrp) \cdot \vv}\\
& \qquad
+\inttauO{(1-\alpha)\bvrm\Grad \Hm'(\bvrm) \cdot \vv} , 
\end{align*}
due to equality of the pressures, and similarly
\begin{align*}
&
\intO{\lr{\alpha\Pp(\bvrp)+(1-\alpha)\Pm(\bvrm)}}\Big|_0^\tau
\\
&\quad
=\inttauO{
\Big( 
\alpha \pt \Pp(\bvrp) + (1-\alpha) \pt \Pm(\bvrm)
\Big) } 
\\
& \quad 
=\inttauO{\alpha\bvrp\pt \Hp'(\bvrp)}\\
& \qquad 
+\inttauO{(1-\alpha)\bvrm\pt \Hm'(\bvrm)}.
\end{align*}
Substituting for these terms into \eqref{sumstep1b}, we obtain
\eq{      \label{sumstep1c}
&\bar{\cal E}(R,Q,\vu\,| \,\BR,\BQ,\vv)\Big|_0^\tau+\inttauO{\vS(\Grad\vu-\Grad\vv):\Grad(\vu-\vv)}\\
&\leq \inttauO{ \Big(
(R+Q)(\vv-\vu)\cdot\lr{\pt\vv+(\vu \cdot \Grad) \vv}+ \vS(\Grad\vv):\Grad(\vv-\vu)
\Big)
}\\
&+\inttauO{\alpha(\bvrp-\vrp)\pt\Hp'(\bvrp)}\\
&+\inttauO{(1-\alpha)(\bvrm-\vrm)\pt\Hm'(\bvrm)}\\
& +\inttauO{ \Big(\alpha\Pp(\bvrp)+(1-\alpha)\Pm(\bvrm)-p \Big)\Div\vv}\\
&+\inttauO{\alpha(\bvrp\vv-\vrp\vu)  \cdot \Grad\Hp'(\bvrp)}\\
&+\inttauO{(1-\alpha)(\bvrm\vv-\vrm\vu)  \cdot \Grad\Hm'(\bvrm)}\\
&=:I+II+III+\ldots+VI. 
}

This is an analogous form of the relative entropy inequality as the one derived in Theorem 5.2 from \cite{JN19}. Let us stress that up until now we did not use any properties of $(\beta,\bvrp,\bvrm,\vv)$ other than $\Pp(\bvrp)=\Pm(\bvrm)$ and the sufficient regularity in order to use $(\beta,\bvrp,\bvrm,\vv)$ as test functions.

\section{Relative entropy inequality for the strong solution}\label{Sec:4}
We will now restrict our class of test functions to $(\beta,\bvrp,\bvrm,\vv)$ given in Theorem \ref{t:local}, in particular satisfying \eqref{S}. We reformulate and reduce each of the terms from \eqref{sumstep1c}. Then
\begin{align*}
I=&\inttauO{\left[(R+Q-\BR-\BQ)\pt\vv+\lr{(R+Q)\vu-(\BR+\BQ)\vv}\cdot\Grad\vv\right]\cdot(\vv-\vu)}\\
&-\inttauO{(\vv-\vu)\cdot\Grad(\beta\Pp(\bvrp)+(1-\beta)\Pm(\bvrm))}\\
=:&\, I_1+I_2.
\end{align*}
Moreover, we can write
\begin{align*}
&\pt\Hp'(\bvrp)=\Hp''(\bvrp)\pt\bvrp\\
&=-\Hp''(\bvrp)\Div(\bvrp\vv)
+\bvrp \Hp''(\bvrp)\frac{ \gamma-1  }{ \gamma(1-\beta) +\beta }  (1-\beta)
\Div \vv\\
&=-\Grad\Hp'(\bvrp) \cdot \vv-\Pp'(\bvrp)\Div\vv
+\Pp'(\bvrp)\frac{ \gamma-1  }{ \gamma(1-\beta) +\beta }  (1-\beta)
\Div \vv,
\end{align*}
and similarly
\begin{align*}
&\pt\Hm'(\bvrm)=\Hm''(\bvrm)\pt\bvrm\\
&=-\Hm''(\bvrm)\Div(\bvrm\vv)
-\bvrm \Hm''(\bvrm)\frac{ \gamma-1  }{ \gamma(1-\beta) +\beta }  \beta
\Div \vv\\
&=-\Grad\Hm'(\bvrm) \cdot \vv-\Pm'(\bvrm)  \Div\vv
-\Pm'(\bvrm)\frac{ \gamma-1  }{ \gamma(1-\beta) +\beta } \beta
\Div \vv.
\end{align*}
Therefore, the terms $II$ and $III$ can be rewritten as
\begin{align*}
&\inttauO{\alpha(\bvrp-\vrp)\pt\Hp'(\bvrp)}\\
& \quad 
=
-\inttauO{\alpha(\bvrp-\vrp)\vv\cdot\Grad\Hp'(\bvrp)}\\
&
\qquad
-\inttauO{\alpha(\bvrp-\vrp)\Pp'(\bvrp)\Div\vv}\\
&
\qquad 
+\inttauO{\alpha(\bvrp-\vrp)\Pp'(\bvrp)\frac{ \gamma-1  }{ \gamma(1-\beta) +\beta }  (1-\beta)
\Div \vv}\\
&\quad 
=: 
II_1+II_2+II_3
\end{align*}
and
\begin{align*}
&\inttauO{(1-\alpha)(\bvrm-\vrm)\pt\Hm'(\bvrm)}\\
&\quad 
=-\inttauO{(1-\alpha)(\bvrm-\vrm)\vv\cdot\Grad\Hm'(\bvrm)}\\
&\qquad 
-\inttauO{(1-\alpha)(\bvrm-\vrm)\Pm'(\bvrm)\Div\vv}\\
&\qquad 
-\inttauO{(1-\alpha)(\bvrm-\vrm)\Pm'(\bvrm)\frac{ \gamma-1  }{ \gamma(1-\beta) +\beta }  \beta
\Div \vv}\\
& \quad 
=:
III_1+III_2+III_3.
\end{align*}
This means that we have
\begin{align*}
&IV+II_2+III_2=\inttauO{
\alpha
\Big(
\Pp(\bvrp)-\Pp'(\bvrp)(\bvrp-\vrp)-\Pp(\vrp)
\Big)
\Div\vv}\\
&\qquad\qquad\qquad \qquad
+\inttauO{(1-\alpha)
\Big(
\Pm(\bvrm)-\Pm'(\bvrm)(\bvrm-\vrm)-\Pm(\vrm)
\Big)
\Div\vv},\\
&V+II_1=\inttauO{\alpha\vrp(\vv-\vu) \cdot \Grad \Hp'(\bvrp)},\\
&VI+III_1=\inttauO{(1-\alpha)\vrm(\vv-\vu) \cdot \Grad \Hm'(\bvrm)}.
\end{align*}
We can now simplify further by writing
\begin{align*}
V+II_1+VI+&III_1+I_2\\
&= \inttauO{(\vv-\vu) \cdot \Grad \Hp'(\bvrp)(\alpha\vrp-\beta\bvrp)}\\
& \quad 
+\inttauO{(\vv-\vu) \cdot \Grad \Hm'(\bvrm)
\Big(
(1-\alpha)\vrm-(1-\beta)\bvrm
\Big)
}\\
& \quad 
-\inttauO{(\vv-\vu) \cdot \Big(
\Grad\beta \Pp(\bvrp)+\Grad(1-\beta)\Pm(\bvrm)
\Big)
} . 
\end{align*}

In summary, we have shown that the right-hand side of \eqref{sumstep1c} can be rewritten as a remainder ${\cal R}$ which being a sum of eight integrals. More precisely,
\eq{ \label{remainder_c}
&{\cal R}
=\inttauO{\left[(R+Q-\BR-\BQ)\pt\vv+\lr{(R+Q)\vu-(\BR+\BQ)\vv}\cdot\Grad\vv\right]\cdot(\vv-\vu)}\\
&+\inttauO{\alpha
\Big(\Pp(\bvrp)-\Pp'(\bvrp)(\bvrp-\vrp)-\Pp(\vrp) \Big) \Div\vv}\\
&+\inttauO{(1-\alpha) \Big(\Pm(\bvrm)-\Pm'(\bvrm)(\bvrm-\vrm)-\Pm(\vrm) \Big) \Div\vv}
\\ 
&+\inttauO{(\vv-\vu) \cdot \Grad \Hp'(\bvrp)(\alpha\vrp-\beta\bvrp)}\\
&+\inttauO{(\vv-\vu) \cdot \Grad \Hm'(\bvrm)
\Big( (1-\alpha)\vrm-(1-\beta)\bvrm \Big)
}\\
&-\inttauO{(\vv-\vu) \cdot \Big( \Grad\beta \Pp(\bvrp)+\Grad(1-\beta)\Pm(\bvrm) \Big)
}\\
&+\inttauO{\alpha(\bvrp-\vrp)\Pp'(\bvrp)\frac{ \gamma-1  }{ \gamma(1-\beta) +\beta }  (1-\beta)
\Div \vv}\\
&-\inttauO{(1-\alpha)(\bvrm-\vrm)\Pm'(\bvrm)\frac{ \gamma-1  }{ \gamma(1-\beta) +\beta }  \beta
\Div \vv}\\
&=: J_1+\ldots +J_{8} . 
}

\section{The weak--strong uniqueness principle}\label{Sec:5}
We first notice that
\begin{align} \label{res_ess}
 &  \bar{\cal E}(R,Q,\vu\,| \,\BR,\BQ,\vv)(\tau)  \nonumber   \\
 & \qquad 
   \geq C \begin{cases}
        &\displaystyle \intO{  \Big[
        \alpha\lr{\vrp-\bvrp}^2  +(1-\alpha)\lr{\vrm-\bvrm}^2   \Big]
        }, \quad \vrp,\vrm \in \left[c_\star, c^\star\right],  \\[3ex] 
         & \displaystyle   \intO{ 
         \lr{1+\alpha\vrp^\gp+(1-\alpha)\vrm^\gm}
         } ,\quad  \text{ otherwise},
    \end{cases}
\end{align}
for any $c_\star,c^\star>0$ and $\tau\in(0,T)$.
Motivated by this observation, we also define the essential and residual sets corresponding to $\vrp,\vrm$ as follows
\begin{align*}
\Omega_{ \rm ess}(t):=\{x\in\Omega: \vrp,\vrm \in \left[c_\star, c^\star\right]\},
\quad 
\Omega_{\rm res}(t)=\Omega\backslash \Omega_{ \rm ess}(t).
\end{align*}

Let us now more precisely formulate our main theorem comparing global weak solutions in the time interval $(0,T)$ with the local strong solutions.

\begin{thm}\label{thm:entropy}
Let $(\alpha , \vr_-,\vr_+,\vu)$ be a weak solution emanating from the data $(\alpha_0 , \vr_{-,0},\vr_{+,0},\vu_0)$, given by Theorem \ref{t1bi} and let $(\beta, \bvrm,\bvrp,\vv)$ be the strong solution emanating from the data $(\beta_0 , \tilde\vr_{-,0},\tilde\vr_{+,0},\vv_0)$. 
Then there exists a positive constant $C$ depending on the strong solution (but not on the regularity of the weak solution), such that for a.a. $\tau\in(0,T)$
\begin{align*}
  &{\cal E}(\alpha,R,Q,\vu \,| \, \beta,\BR,\BQ,\vv)(\tau)+\inttauO{\vS(\Grad\vu-\Grad\vv):\Grad(\vu-\vv)}   \\
  &\qquad\qquad\qquad
  \leq C 
  {\cal E}(\alpha,R,Q,\vu \,| \, \beta,\BR,\BQ,\vv)(0),
\end{align*}
where ${\cal E}(\alpha,R,Q,\vu \,| \, \beta,\BR,\BQ,\vv)$ is given in \eqref{rel_entr}.

In particular, if $(\beta_0 , \tilde\vr_{-,0},\tilde\vr_{+,0},\vv_0)=(\alpha_0 , \vr_{-,0},\vr_{+,0},\vu_0)$, then 
\begin{align*}
(\alpha , \vr_-,\vr_+,\vu)(\tau,x)=(\beta, , \bvrm,\bvrp,\vv)(\tau,x)
\end{align*}
for a.a. $\tau \in (0,T),\,x\in \Omega$.
\end{thm}

Before proving Theorem \ref{thm:entropy}, we will state and prove an auxiliary result targeting specifically  the difference of the volumetric fractions $\alpha$ and $\beta$.

\begin{lemma}\label{lemm:alpha}
Let $(\alpha,\vu)$, $(\beta,\vv)$ satisfying 
\begin{align*}
\alpha,\beta\in L^\infty((0,T)\times\Omega)\cap C([0,T];L^1(\Omega)),\quad \vu,\vv\in L^2(0,T; W^{1,2}_0(\Omega;\R^3))
\end{align*}
be two distributional solutions to the transport equation \eqref{equa_alpha}. Suppose, in addition, that
\begin{align*}
\Grad \beta\in L^\infty((0,T)\times\Omega;\R^3),\quad \Div\vv\in L^1(0,T; L^\infty(\Omega))
\end{align*}
and $0\leq \alpha, \beta \leq 1$ almost everywhere in $(0,T)\times \Omega$.
Then, for any $\delta>0$ and for all $\tau\in [0,T]$, we have
\eq{   \label{al-be}
\intO{(\alpha-\beta)^2}\Big|_0^\tau\leq \delta\int_0^\tau \|\vv-\vu\|^{2}_{W^{1,2}(\Omega)}\,\dt 
+
\int_0^\tau   C_\delta(t)  \intO{(\alpha-\beta)^2}\,\dt.
}
with $C_\delta(t)$ integrable over $(0,T)$.
\end{lemma}

\begin{pf}
 To rigorously justify this lemma, one may use the DiPerna-Lions regularizing technique \cite{DiPL}, see also Lemma 7.3 in \cite{JN19} for the same argument applied to the pure transport equations for $\alpha,\beta$. In our case, the equations of $\alpha,\beta$ contain additional nonlinear terms, so the analysis is more involved. 
We have
\begin{align*}
&\pt(\alpha-\beta)+\vv\cdot\Grad(\alpha-\beta)\\
&\qquad \qquad  =(\vv-\vu) \cdot \Grad\alpha -\frac{ (\gamma-1) (1-\alpha) }{ \gamma(1-\alpha) +\alpha } \alpha 
\Div \vu+\frac{ (\gamma-1)(1-\beta)  }{ \gamma(1-\beta) +\beta } \beta
\Div \vv.
\end{align*}
Multiplying this equation by $(\alpha-\beta)$, we obtain ($\omega_\alpha =\frac{ (\gamma-1) (1-\alpha) }{ \gamma(1-\alpha) +\alpha } \alpha$, similarly $\omega_\beta$)
\begin{align*}
&\intO{\frac12\partial_t(\alpha-\beta)^2}-\frac12\intO{\Div\vv(\alpha-\beta)^2}\\
& \quad 
=\intO{(\vv-\vu) \cdot \Grad\alpha(\alpha-\beta)}\\
& \qquad 
+\intO{(\beta-\alpha)(\omega_\alpha\Div\vu-\omega_\beta\Div\vv)}\\
& \quad 
=-\frac{1}{2}\intO{\Div(\vv-\vu)(\alpha-\beta)^2}\\
& \qquad 
+\intO{(\vv-\vu) \cdot \Grad\beta(\alpha-\beta)}\\
&\qquad 
+\intO{(\beta-\alpha)(\omega_\alpha-\omega_\beta)\Div\vv}\\
&\qquad 
+\intO{(\beta-\alpha)\omega_\alpha(\Div\vu-\Div\vv)}.
\end{align*}
A straightforward calculation shows 
\begin{align*}
|\omega_\alpha-\omega_\beta|\leq C|\alpha-\beta|,
\end{align*}
where $C=C(\|\alpha\|_{L^\infty((0,T)\times\Omega)}, \|\beta\|_{L^\infty((0,T)\times\Omega)})>0$. Indeed,
\begin{align*}
 \omega_\alpha'=
 \frac{ \gamma-1  }{ \gamma(1-\alpha) +\alpha } (1-\alpha) 
 - \frac{ \gamma-1  }{ \gamma(1-\alpha) +\alpha } \alpha 
 - \alpha (1-\alpha)  
 \frac{ \gamma-1  }{  [\gamma(1-\alpha) +\alpha]^2 } (-\gamma+1), 
\end{align*}
yielding 
\begin{align*}
| \omega_\alpha'| \leq 
2 \f{\gamma+1}{\min\{ \gamma,1 \} }
+
\f{ (\gamma+1)^2  }{   ( \min\{ \gamma,1 \})^2  }.
\end{align*}
We may choose
$\delta>0$ arbitrarily small so that
\begin{align*}
&\intO{\frac12\partial_t(\alpha-\beta)^2}\\
&\quad \leq\delta\|\Div(\vv-\vu)\|^2_{L^2(\Omega)}+\delta\|\vv-\vu\|^2_{L^2(\Omega)}\\
&\quad +C_\delta\lr{\|\alpha-\beta\|_{L^\infty(\Omega)}^2+\|\Grad\beta\|_{L^\infty(\Omega)}^2+\|\Div \vv\|_{L^\infty(\Omega)}}\intO{(\alpha-\beta)^2},
\end{align*}
where we also used the crucial fact that $\omega_\alpha$ is uniformly bounded. 
The statement thus follows by integrating the above inequality with respect to $t\in (0,\tau)$.  $\Box$
\end{pf}

\bigskip

With Lemma \ref{lemm:alpha} at hand, we may proceed to the proof of Theorem \ref{thm:entropy}. We will prove it by estimating $J_1,\ldots, J_{8}$ term by term.

\medskip 

\noindent{\emph{Estimate of $J_1$}.} We split the first term of the remainder \eqref{remainder_c} as follows
\begin{align*}
J_1=&\inttauO{\left[(R+Q-\BR-\BQ)\pt\vv+\lr{(R+Q)\vu-(\BR+\BQ)\vv}\cdot\Grad\vv\right]\cdot(\vv-\vu)}\\
=&\inttauO{\left[(R-\BR+Q-\BQ)(\pt\vv+ \vv\cdot\Grad\vv)\right]\cdot(\vv-\vu)}\\
&+\inttauO{(R+Q)(\vu-\vv)\cdot\Grad\vv\cdot(\vv-\vu)}.
\end{align*}
The second term can be estimated by $ \int_0^\tau C(t) \bar{\cal E}(R,Q,\vu|\BR,\BQ,\vv)\,\dt$ with  $C(t)$ depending on $\|\Grad\vv\|_{L^\infty(\Omega)}(t)$.
To estimate the first term, we use
\eq{\label{decomp}
R-\BR=\alpha(\vrp-\bvrp)-\bvrp(\beta-\alpha),
}
and similarly for $Q-\BQ$. Thus, together, we obtain
\begin{align*}
|R-\BR+Q-\BQ|\leq C
\Big(
\sqrt{\alpha}|\vrp-\bvrp|+\sqrt{1-\alpha}|\vrm-\bvrm|+|\beta-\alpha| 
\Big),
\end{align*}
where $C =C(t)>0$ depends only on $\|(\bvrp,\bvrm)\|_{L^\infty(\Omega)}(t)$ and on $\|\alpha\|_{L^\infty(\Omega)}(t)$.
Therefore,
\begin{align*}
&\inttauO{\mathbf{1}_{ \rm ess}\left[(R-\BR+Q-\BQ)(\pt\vv+\vv\cdot\Grad\vv)\right]\cdot(\vv-\vu)}\\
&\quad \leq \delta \int_0^\tau\|\vv-\vu\|_{L^2(\Omega)}^2\,\dt +\int_0^\tau C_\delta(t)  \Big(\bar{\cal E}(R,Q,\vu|\BR,\BQ,\vv)+\|\alpha-\beta\|_{L^2(\Omega)}^2 \Big)\,\dt,
\end{align*}
with $C_\delta(t)$ proportional to $\|\pt\vv+\vv\cdot \nabla\vv\|^2_{L^\infty(\Omega)}(t)$. For the residual part of this term in $J_1$, considering the small and large values of $\vrm,\vrp$ separately, we get
\begin{align*}
&\inttauO{\mathbf{1}_{ \rm res}\left[(R-\BR+Q-\BQ)(\pt\vv+\vv\cdot\Grad\vv)\right]\cdot(\vv-\vu)}\\
&\quad\leq \delta \int_0^\tau\|\vv-\vu\|_{L^2(\Omega)}^2\,\dt +\int_0^\tau C_\delta(t) \bar{\cal E}(R,Q,\vu|\BR,\BQ,\vv)\,\dt\\
&\qquad +\int_0^\tau \|\pt\vv+\vv\cdot \nabla \vv\|_{L^\infty(\Omega)}\|R|_{ \rm res}, Q|_{ \rm res}\|_{L^1(\Omega)}^{\frac{1}{2}}\|(R+Q)|\vu-\vv|^2\|_{L^1(\Omega)}^{\frac{1}{2}} \, \dt \\
&\quad \leq \delta \int_0^\tau\|\vv-\vu\|_{L^2(\Omega)}^2\,\dt +\int_0^\tau C_\delta(t) \bar{\cal E}(R,Q,\vu|\BR,\BQ,\vv)\,\dt,
\end{align*}
with $C_\delta(t)$ proportional to $ \|\pt\vv+\vv\cdot \nabla \vv\|_{L^\infty(\Omega)}(t)$. Ultimately, we have
\begin{align*}
J_1\leq \delta \int_0^\tau\|\vv-\vu\|_{L^2(\Omega)}^2\,\dt +\int_0^\tau C_\delta(t) \bar{\cal E}(R,Q,\vu|\BR,\BQ,\vv)\,\dt,
\end{align*}
where $C_\delta(t)$ depends on $\|\pt\vv+\vv\cdot\Grad\vv\|_{L^\infty(\Omega)}(t)$ and  on $\|\Grad\vv\|_{L^\infty(\Omega)}(t)$ so that $C_\delta(t)$ is integrable over $(0,T)$.

\medskip 

\noindent{\emph{Estimate of $J_2$ and $J_3$}.} Again, both terms will be considered on the essential and residual parts of the domain. Using Taylor expansion for $J_2$, we obtain
\begin{align*}
&\inttauO{\mathbf{1}_{ \rm ess}\alpha 
\Big(
\Pp(\bvrp)-\Pp'(\bvrp)(\bvrp-\vrp)-\Pp(\vrp)
\Big)
\Div\vv}\\
&\leq C\inttauO{\mathbf{1}_{ \rm ess}\alpha|\bvrp-\vrp|^2}
\leq \int_0^\tau C(t) \bar{\cal E}(R,Q,\vu|\BR,\BQ,\vv)\,\dt,
\end{align*}
where $C(t)$ is proportional to $\|\alpha\|_{L^\infty(\Omega)}(t) \|\Div\vv\|_{L^\infty(\Omega)}(t)$. The estimate for the essential part of $J_3$ is similar. The residual part is bounded directly from \eqref{res_ess} and $C(t)$ is again integrable. In summary, we have
\begin{align*}
J_2+J_3\leq  \int_0^\tau C(t) \bar{\cal E}(R,Q,\vu|\BR,\BQ,\vv)\,\dt.
\end{align*}

\medskip 

\noindent{\emph{Estimate of $J_4$ and $J_5$}.} Both terms are again treated in an analogous way, so we estimate only one of them. Using the decomposition \eqref{decomp}, we easily verify that
\begin{align*}
J_4&=\inttauO{(\vv-\vu)\cdot\Grad \Hp'(\bvrp)(\alpha\vrp-\beta\bvrp)}\\
& \leq \delta \int_0^\tau\|\vv-\vu\|_{L^2(\Omega)}^2  \,\dt 
+\int_0^\tau C_\delta(t) \Big(\bar{\cal E}(R,Q,\vu|\BR,\BQ,\vv)+\|\alpha-\beta\|_{L^2(\Omega)}^2 \Big)\,\dt,
\end{align*}
where $C_\delta(t)$ depends on $\|\bvrp\|_{W^{1,\infty}(\Omega)}(t)$. Moreover,
\begin{align*}
&J_5\leq \delta \int_0^\tau\|\vv-\vu\|_{L^2(\Omega)}^2 \, \dt
+\int_0^\tau C_\delta(t) \Big(\bar{\cal E}(R,Q,\vu|\BR,\BQ,\vv)+\|\alpha-\beta\|_{L^2(\Omega)}^2 \Big)\,\dt,
\end{align*}
with $C_\delta(t)$ depending on $\|\bvrm\|_{W^{1,\infty}(\Omega)}(t)$.

\medskip

\noindent{\emph{Estimate of $J_6$}.} This term vanishes thanks to the algebraic closure $\Pp(\bvrp)=\Pm(\bvrm)$. Indeed, we have
\begin{align*}
J_{6}=&-\inttauO{(\vv-\vu)\cdot 
\Big(
\Grad\beta \,\Pp(\bvrp)+\Grad(1-\beta) \,\Pm(\bvrm)
\Big)
}\\
=&-\inttauO{(\vv-\vu)\cdot 
\Big(
\Grad\beta \,  \Pp(\bvrp)-\Grad\beta  \, \Pm(\bvrm)
\Big)
}\\
=& \, 0.
\end{align*}


\medskip

\noindent{\emph{Estimate of $J_7$ and $J_{8}$}.} To estimate  the remaining terms, we couple them with terms 
\eqh{-\inttauO{\alpha\Pp(\bvrp)\frac{ (\gamma-1) (1-\alpha)  }{ \gamma(1-\alpha) +\alpha } 
\Div \vv}}
and
\eqh{\inttauO{\alpha\Pm(\bvrm)\frac{ (\gamma-1) (1-\alpha)  }{ \gamma(1-\alpha) +\alpha } 
\Div \vv},}
 that cancel each other out due to the equality of the pressures $\Pp(\bvrp)=\Pm(\bvrm)$.
 
At first, we transform
\begin{align*}
J_{7}&-\inttauO{\alpha\Pp(\bvrp)\frac{ (\gamma-1) (1-\alpha)  }{ \gamma(1-\alpha) +\alpha } 
\Div \vv}\\
=&\inttauO{\alpha(\bvrp-\vrp)\Pp'(\bvrp)\frac{ (\gamma-1)(1-\beta)  }{ \gamma(1-\beta) +\beta }  
\Div \vv}\\
&-\inttauO{\alpha\Pp(\bvrp)\frac{ (\gamma-1) (1-\alpha)  }{ \gamma(1-\alpha) +\alpha } 
\Div \vv}\\
=&-\inttauO{\alpha\frac{ (\gamma-1)(1-\alpha)  }{ \gamma(1-\alpha) +\alpha } \left[\Pp(\bvrp)-(\bvrp-\vrp)\Pp'(\bvrp)\right]\Div\vv} \\
&-\inttauO{\alpha\lr{\frac{ (\gamma-1)(1-\alpha)  }{ \gamma(1-\alpha) +\alpha } -\frac{ (\gamma-1) (1-\beta) }{ \gamma(1-\beta) +\beta} }(\bvrp-\vrp)\Pp'(\bvrp)\Div\vv}\\
=:&\, L_1+L_2.
\end{align*}
Notice that the last term $L_2$ can already be estimated, because the difference in brackets is controlled by $C|\alpha-\beta|$, with $C>0$ depending on $\|(\alpha,\beta)\|_{L^\infty(\Omega)}$. Thus, we have
\begin{align*}
L_2\leq \int_0^\tau C(t) \Big(\bar{\cal E}(R,Q,\vu\,| \,\BR,\BQ,\vv)+\|\alpha-\beta\|_{L^2(\Omega)}^2 \Big)\,\dt,
\end{align*}
with $C(t)>0$ depending additionally on $\|\bvrp\|_{L^\infty(\Omega)}(t)$, and $\|\Div\vv\|_{L^\infty(\Omega)}(t)$.
Note that $C(t)$ is integrable over $(0,T)$.

We now transfer in a similar way
\begin{align*}
J_{8}&+\inttauO{\alpha\Pm(\bvrm)\frac{ (\gamma-1) (1-\alpha)  }{ \gamma(1-\alpha) +\alpha } 
\Div \vv}\\
=&-\inttauO{(1-\alpha)(\bvrm-\vrm)\Pm'(\bvrm)\frac{ (\gamma-1)\beta  }{ \gamma(1-\beta) +\beta }  
\Div \vv}\\
&+\inttauO{\alpha\Pm(\bvrm)\frac{ (\gamma-1) (1-\alpha)  }{ \gamma(1-\alpha) +\alpha } 
\Div \vv}\\
=&\inttauO{\alpha\frac{ (\gamma-1)(1-\alpha)  }{ \gamma(1-\alpha) +\alpha }  [\Pm(\bvrm)-(\bvrm-\vrm)\Pm'(\bvrm)]\Div\vv}\\
& +\inttauO{(1-\alpha)\lr{\frac{ (\gamma-1)\alpha  }{ \gamma(1-\alpha) +\alpha }  -\frac{ (\gamma-1)\beta  }{ \gamma(1-\beta) +\beta }  }(\bvrm-\vrm)\Pm'(\bvrm)\Div\vv}\\
=:& \, {\emph{\L}}_1+{\emph{\L}}_2.
\end{align*}
The term ${\emph{\L}}_2$ can be estimated by the same way as $L_2$, leading to
\begin{align*}
{\emph{\L}}_2\leq \int_0^\tau C(t) \Big(\bar{\cal E}(R,Q,\vu\,| \,\BR,\BQ,\vv)+\|\alpha-\beta\|_{L^2(\Omega)}^2 \Big)\,\dt,
\end{align*}
with $C(t)$ depending on $\|\bvrm\|_{L^\infty(\Omega)}(t)$ and $\|\Div\vv\|_{L^\infty(\Omega)}(t)$ so that it is integrable.

Finally, let us look at the terms $L_1$ and ${\emph{\L}}_1$. Using again the equality of pressures, we can write
\begin{align*}
&L_1+{\emph{\L}}_1 \\
&\quad= -\inttauO{\alpha\frac{ (\gamma-1)(1-\alpha)  }{ \gamma(1-\alpha) +\alpha } \Div \vv \Big( \Pp(\bvrp)-(\bvrp-\vrp)\Pp'(\bvrp)-\Pp(\vrp) \Big)
}\\
& \qquad \,\,
+\inttauO{\alpha\frac{ (\gamma-1)(1-\alpha)  }{ \gamma(1-\alpha) +\alpha } \Div \vv \Big(  \Pm(\bvrm)-(\bvrm-\vrm)\Pm'(\bvrm)-\Pm(\vrm)  \Big)
},
\end{align*}
and these two integrals can be treated exactly as the terms $J_2,\ J_3$ before.
Summarizing, we have shown that
\begin{align*}
J_7+J_{8}=&\, 
L_1+L_2+{\emph{\L}}_1+{\emph{\L}}_2\\
\leq & \,
\int_0^\tau C(t) \Big(\bar{\cal E}(R,Q,\vu\,| \,\BR,\BQ,\vv)+\|\alpha-\beta\|_{L^2(\Omega)}^2 \Big)\,\dt.
\end{align*}

\medskip

\noindent{\emph{Conclusion}.} Summarizing the above estimates of the remainder ${\cal R}$ given by \eqref{remainder_c} and coming back to \eqref{sumstep1c}, we obtain
\begin{align*}
&\bar{\cal E}(R,Q,\vu\,| \,\BR,\BQ,\vv)\Big|_0^\tau+\inttauO{\vS(\Grad\vu-\Grad\vv):\Grad(\vu-\vv)}\\
& \quad 
\leq
\delta \int_0^\tau\|\vv-\vu\|_{W^{1,2}(\Omega)}^2\,\dt+ \int_0^\tau C_\delta(t)
\Big(
\bar{\cal E}(R,Q,\vu\,| \,\BR,\BQ,\vv)+\|\alpha-\beta\|_{L^2(\Omega)}^2
\Big)\,\dt,
\end{align*}
for $C_\delta(t)$ depending on $\|\pt\vv+\vv\cdot \nabla\vv\|_{L^\infty(\Omega)}(t)$, $\|\Grad\vv\|_{L^\infty(\Omega)}(t)$, $\|\bvrp\|_{W^{1,\infty}(\Omega)}(t)$ and $\|\bvrm\|_{W^{1,\infty}(\Omega)}(t)$,  as well as on the upper limits of $\alpha$ and $\beta$, which are equal to $1$, so that $C_\delta$ is integrable over $(0,T)$.

Summing up this estimate with \eqref{al-be} and using the Gr\"{o}nwall argument, we finish the proof of Theorem \ref{thm:entropy}. In particular, our main Theorem \ref{thm_entropy} follows. $\Box$

\bigskip

\noindent{\bf{Acknowledgement.}} The work of Y.L. was supported by National Natural Science Foundation of China (12571228), Natural Science Foundation of Anhui Province (2408085MA018). 
 The work of M.L.-M. was supported by the Gutenberg Research College and by the Deutsche Forschungsgemeinschaft (DFG, German Research Foundation)--project number 233630050--TRR 146 and project number 525853336--SPP 2410 ``Hyperbolic Balance Laws: Complexity, Scales and Randomness". She is also grateful to the Mainz Institute of Multiscale Modelling for supporting her research. The work of M.P. was partially supported by the Czech Science Foundation (GA\v{C}R), project No. 25-16592S. The work of E.Z. was  supported by the EPSRC Early Career Fellowship no. EP/V000586/1.

\vspace{4mm}
\noindent{\bf{Data Availability.}} Data sharing is not applicable to this article as no datasets were generated or analyzed
during the current study.
\vspace{4mm}

\noindent{\bf{Conflicts of interest.}} All authors certify that there are no conflicts of interest for this work.

\bigskip

\noindent{\bf{Publishing licence.}} For the purpose of open access, the author has applied a Creative Commons Attribution (CC BY) licence to any Author Accepted Manuscript version arising from this submission.

\end{document}